\documentclass[a4paper,12pt]{amsart}
\usepackage{a4wide,graphicx,amssymb}
\usepackage{tikz}
\usetikzlibrary{decorations.pathreplacing,arrows}
\usetikzlibrary{shapes}
 \usepackage{enumerate}
 \usepackage{multirow}
 \usepackage{multicol}

 \newcommand{\at}{\makeatletter @\makeatother}           

\newcommand{\mast}{\operatorname{mast}}
\newcommand{\tw}{\operatorname{tw}}

\newif\ifdetails
\detailstrue
\newcommand{\DETAIL}[1]%
{\ifdetails\par\fbox{\begin{minipage}{0.9\linewidth}\textit{Detail:}
      #1\end{minipage}}\par\fi}
\newcommand{\TODO}[1]%
{\ifdetails\par\fbox{\begin{minipage}{0.9\linewidth}\textbf{TODO:}
      #1\end{minipage}}\par\fi}

\newtheorem{lemma}{Lemma}

\newtheorem{theorem}[lemma]{Theorem}

\definecolor{changecolor}{RGB}{192,64,0}

\usepackage[normalem]{ulem}

\newcommand{\old}[1]{{}}

\title{On agreement subtrees in multiple phylogenetic trees}

\author{Bryan Currie} 
\author{\'Eva Czabarka}
\author{Mareike Fischer}
\author{Henry Simmons}
\author{L\'aszl\'o Sz\'ekely}
\author{Kristina Wicke} 

\address{Bryan Currie \\ Department of Mathematical Sciences\\ New Jersey Institute of Technology\\ Newark, NJ\\ USA}
\email{bc479@njit.edu}
\address{\'Eva Czabarka\\Department of Mathematics \\ University of South Carolina \\ Columbia SC 29212 \\ USA}
\email{czabarka@math.sc.edu}
\address{Mareike Fischer\\Institute of Mathematics and Computer Science\\ University of 
Greifswald\\ Greifswald, Germany}
\email{mareike.fischer@uni-greifswald.de}
\address{Henry Simmons\\Department of Mathematics \\ University of South Carolina \\ Columbia SC 29212 \\ USA}
\email{has11@email.sc.edu}
\address{L\'aszl\'o Sz\'ekely\\Department of Mathematics \\ University of South Carolina \\ Columbia SC 29212 \\ USA}
\email{szekely@math.sc.edu}
\address{Kristina Wicke\\ Department of Mathematical Sciences\\ New Jersey Institute of Technology\\ Newark, NJ\\ USA\\ and National Institute for Theory and Mathematics in Biology\\ Northwestern University and The University of Chicago\\ Chicago, IL\\ USA }
\email{kristina.wicke@njit.edu}

\subjclass[2020]{Primary 05C05; secondary 05D10}
\keywords{maximum agreement subtree, quartet split, Ramsey theorem,  bipartition}
\thanks{
This material is based upon work supported by the National Science Foundation under Grant No. DMS-1929284 while the authors were in residence at the Institute for Computational and Experimental Research in Mathematics in Providence, RI, during the "Exploring The Graph-Theoretical Properties Of Semi-Directed Phylogenetic Networks" Colloborate{\at}ICERM program.
}

 \numberwithin{equation}{section}

\begin{document}

\begin{abstract}
Snir and Yuster [{\em Discrete Appl. Math.} {\bf 347}(2026) 160--171]
asked for the least number $h(k)$ such that  $k$ unrooted binary phylogenetic trees on the same $h(k)$ leaves always share a common quartet. We give a new upper bound
for the $k$-tree version of the Maximum Agreement Subtree problem, namely an upper bound for
the number of leaves, on which
$k$ unrooted binary phylogenetic trees always share a
common induced binary subtree on $n$ leaves, which is a four-times iterated exponential function. For $h(k)$, this implies a four-times
iterated exponential upper bound. We also set an exponential lower bound for $h(k)$. 
\end{abstract}

\maketitle

\section{Phylogenetic trees}
A {\it phylogenetic tree} is
a binary tree in which the leaves are labeled bijectively with
labels from a set $X$ (usually $X= \{1,2,...,n\}$) and internal vertices
are unlabeled.
Two phylogenetic trees are considered
 the same, if there is a label-preserving
graph isomorphism between them. Note that  phylogenetic trees in this note
are not rooted.

If $T$ is a phylogenetic tree  and $ Y\subseteq X$  is a set of labels,
then the {\it  induced binary subtree }
$T|_Y$ is defined as follows:
(a) take the subtree induced by $Y$ in $T$, and (b)
substitute paths in which all internal vertices have degree 2
 by edges. 
If $|Y|=4$, the induced binary subtree is often identified with an
unordered partition of $Y$ into two two-element sets, obtained by
removing the (unique) internal edge of  $T|_Y$. This partition
is known as a {\it quartet split}. It has long been known that the $
\binom{n}{4}$
quartet splits of a phylogenetic tree with $n$ leaves determine the
phylogenetic tree through a polynomial time algorithm. This was
first observed in 1981 by Colonius and Schultze \cite{colonius},
in the context of stemmatology. The theory of quartet splits was developed further in 1986 by
Bandelt and Dress \cite{ban86}, who 
also noted that 
\begin{equation} \label{qsplit _on_6}
\hbox{ any two phylogenetic trees on the same  $n\geq 6$ leaves share a quartet split.}
\end{equation}
We use this fact several times in this paper.

An important algorithmic problem, known as the {\em Maximum Agreement Subtree
Problem}, is the following: given two phylogenetic trees, both labeled with elements of the same set $X$, find a
common induced binary subtree of the largest possible size.

The  Maximum Agreement Subtree
Problem has a long history, starting with   papers in the early 1980s by Gordon
 \cite{gordon}, and
Finden and Gordon \cite{fgordon}. 
Somewhat surprisingly, this
problem can be solved in polynomial time \cite{cole2000,stewar}.
The  more general 
version of the Maximum Agreement Subtree
Problem has been posed for more than two trees. It still can be solved in polynomial time, which is
linear in the number of trees \cite{Farach}.
 Note, however, that for more than two trees, the problem becomes NP-hard when the phylogenetic trees are allowed to have unbounded degree rather than being restricted to binary trees~\cite{amir1997}.

Here we focus on the extremal version of the problem, where we ask how small a maximum agreement subtree
can possibly be between two phylogenetic trees with the same $n$ leaves. This minimum is denoted 
by $\mast_2(n)$. The order of magnitude of  $\mast_2(n)$ was a long-standing open problem. Finally, in 2020, Alexey Markin~\cite{markin} proved  
\begin{equation} \label{markin}
\mast_2(n)=\Theta(\log n).
\end{equation} Similarly, we define by $\mast_k(n)$ the smallest size
of the Maximum Agreement Subtree of $k$ phylogenetic trees on the same $n$ leaves. It is easy to see
that for $k\ge 3$ we have
\begin{equation} \label{mastrec}
\mast_k(n)\geq \mast_2(\mast_{k-1}(n)).
\end{equation}

For convenience, especially for using results from Ramsey theory, we introduce $M_k(m)$, which is essentially
the inverse of $\mast_k(n)$, as
\begin{equation} \label{Mdef}
M_k(m)=\min\{n: \, \mast_k(n)\geq m\}.
\end{equation}
By the result of Markin (\ref{markin}), $M_2(m)$ grows exponentially, and the obvious 
recursion, which holds for  all $k\geq 3$, is
\begin{equation} \label{Mrec}
M_k(n)\leq M_2(M_{k-1}(n)).
\end{equation}
For a fixed $k$, (\ref{Mrec}) gives a $(k-1)$-times iterated exponential upper bound for $M_k(n)$. 
We are going to get a better upper bound
for $M_k(n)$, which uses only 4-times iterated exponentials;
and hence a better lower bound for $\mast_k(n)$, which uses only 4-times iterated logarithms. This gives an immediate extension of the range of Theorem 7 of Czabarka, Kelk, Moulton and Sz\'ekely  in \cite{coconvex}, using 4-times iterated logarithms instead of $(t-1)$-times iterated logarithms,  as follows:

\begin{theorem}  
 For every fixed  $t\ge 2$, 
 for some $c_t>0$,
 for every sufficiently large $n$, for all integers $k$ with $n-c_t\log\log\log\log n \leq  k\leq n-4$,
 for any $t$ phylogenetic trees with the same $n$
 leaves, there exists a  partition of the common leaf set into $k$ nonempty classes, such that
 \begin{itemize}
 \item the partition has at least two classes with at least two elements, and 
 \item in each of the $k$ trees, the partition of the leaves appears in such a way that after the removal of some edges of the tree, the partition classes are exactly the leaf sets of the remaining connected components. 
\end{itemize}
 \end{theorem}

\section{An improved upper bound on the $M_k(n)$ function}
Recall that for $k,r,a\in\mathbb{Z}^+$, the {\em Ramsey number} $R_k^r(a)$ denotes the smallest integer $m$
such that for every coloring of the $r$-element subsets of any $m$-element set with $k$ colors, there exists an $a$-element subset $A$ of the $m$-element set whose $r$-element subsets are all assigned the same color.  Ramsey's Theorem asserts that these numbers
exist. Steel and Sz\'ekely \cite{our_mast} connected the Maximum Agreement Subtree problem to Ramsey theory  by showing \[{\mast}_2[ R_2^4(n,6) ] \geq n.\]

The following recursive bound from Erd\H os and Rado holds for Ramsey numbers (\cite{ErdosRado}, see also \cite{lovasz}, Ex. 14.3):
\begin{equation} \label{Rrec}
 R_k^{r+1}(a)\leq k^{[ R_k^{r}(a) ]^r }.   
\end{equation}
In particular, 
\begin{equation} \label{r=2}
R_k^2(a)\leq k^{ka}.\end{equation} 
Define the tower functions $\tw_i(\cdot)$ recursively as follows: 
$\tw_0(x)=x$ and for $i\ge 1$,  $\tw_{i}(x)=2^{\tw_{i-1}(x)}$.
We use the notation $\log$ for the base 2 logarithm.

The recurrence relation (\ref{Rrec}) implies by induction, 
starting at $r=2$ and 
using (\ref{r=2}) for the base case, that for every fixed $r$ and all $\ell\geq 3$, we have 
\begin{equation} \label{Ramseybound}
 R_\ell^{r}(a)\leq \tw_{r-1}\big(2a\ell\log \ell +O(\log \log \ell)\big),   
\end{equation}
where the $O(\cdot)$-term is with respect to $\ell\rightarrow\infty$.

\begin{theorem} For $n\geq 6$ and $k\geq 3$, we have
\begin{equation}\label{mainineq}M_k(n)\leq R_{2^{k-1}}^4(n)
    \leq \tw_3\Bigg( 2^k(k-1)n +O(\log k )  \Bigg), 
    \end{equation}
 where the $O(\cdot)$-term is with respect to $k\rightarrow\infty$.   
\end{theorem}

\begin{proof}
Consider $k$ arbitrary phylogenetic trees,
$T_1,T_2,\ldots,T_k$ on the same set of $R^4_{2^{k-1}}(n)$ leaves. We 
color each 4-tuple of the leaf set by a binary vector of length $k-1$, thereby obtaining a $2^{k-1}$-coloring of the 4-tuples of the leaf set. The coordinates of the vectors are 
indexed by the integers $i$ with $2\leq i\leq k$, where the $i$th coordinate corresponds to the tree $T_i$. 

For each $4$-tuple of leaves, $Y = \{a,b,c,d\}$, define its color to be the vector $(c_2, \ldots, c_k)$, where
\[ 
c_i = \begin{cases}
1, &\text{if } T_i|_Y = T_1|_Y \text{ (i.e., $T_i$ and $T_1$ induce the same quartet split on $Y$)}\\
0, &\text{otherwise.}
    \end{cases}
\]

By the choice of the number of leaves, {namely $R^4_{2^{k-1}}(n)$,} there exists a set $A$ of $n$ leaves such that any 4-tuple of leaves in $A$ is colored by the same color vector $(c_2,\ldots,c_k)$. 

We claim that this color vector consists entirely of $1$'s, i.e., $c_i=1$ for every $2 \leq i \leq k$. 
To see this, let $i$ be arbitrary with $2\le i\le k$. As the $4$-tuples in $A$ are all colored $(c_2,\ldots,c_k)$, to show $c_i=1$ it suffices to find a $4$-subset of $A$ on which $T_1$ and $T_i$ induce the same quartet split. Such a quartet exists because $|A|=n\ge6$ and by \eqref{qsplit _on_6}.

Thus, $c_i=1$ for all $2\le i\le k$. Hence, every quartet of $A$ is resolved identically in $T_1, \ldots, T_k$. 
Since a phylogenetic tree is uniquely determined by its quartet splits (as discussed in the introduction), the induced subtrees $T_1|_A, \ldots, T_k|_A$ are identical. Therefore the $k$ phylogenetic trees $T_1,\ldots, T_k$ have an agreement subtree on the leaf set $A$. In particular, they have an agreement subtree of size $|A|=n$.

Finally, we apply the estimate in 
(\ref{Ramseybound})
with $r=4$ and  $\ell=2^{k-1}$.
{This gives
\[R_{2^{k-1}}^4(n) \leq \tw_3\left(2n \, 2^{k-1}\log(2^{k-1})+O(\log\log(2^{k-1}))\right).\]
Since $\log(2^{k-1})=k-1$ and $\log\log(2^{k-1})=O(\log k)$, we obtain
\[R_{2^{k-1}}^4(n)\leq\tw_3\left(2^k(k-1)n+O(\log k)\right),\]
as claimed.}
\end{proof}

\section{Multiple  bipartitions}
A recent paper of Sagi Snir and Raphael Yuster~\cite{sagi} studies the superposition of bipartitions on common leaf sets of $k$ trees. In particular, they investigate how many leaves would guarantee $t$ vertices in the intersection of some partition classes coming from the $k$ trees by the removal of an edge, and simultaneously $t$ vertices in the intersection of the complements of the  partition classes. (The paper \cite{sagi} does many other things that we do not discuss here in order to keep the presentation focused.) Let $g(k,n)$ denote the {\em largest} value of $t$ that can be achieved, given any {collection of} $k$ binary $n$-leaf trees on the same leaf set. 

Theorem 1.3 of \cite{sagi} states that
\begin{equation} \label{idezet}
\frac{n}{6}\leq g(2,n) \hbox{\ \ \ and \ \ \ } \frac{2n}{27}\leq g(3,n),
\end{equation}
and that these bounds are asymptotically the best possible.

The first and the second part of (\ref{idezet}) are substantially different. 
The first part of the displayed
formula  (\ref{idezet}) is a special case of an observation of Joel Spencer \cite{spencer} made in the context of 2-planar crossing numbers of random graphs,
where certain lower bounds for crossing numbers hinge on 1/3-2/3 bipartitions of the vertex set of the graph.
We restate Spencer's result in a formulation taken from \cite{asplund}:

\begin{theorem} \label{spencerform}
Let $X$ be an $n$-element set and let $X=X_{11}\cup X_{12}$ and 
$X=X_{21}\cup X_{22}$ 
be any two bipartitions of $X$ such that $\min\{|X_{ij}|:i,j\in \{1,2\}\}\ge n/3$ . Then there exist two subsets $Y_1$  and $Y_2$ of $X$, such that $\min\{|Y_1|,|Y_2|\}\ge n/6$, which lie on different sides of both bipartitions. That is, either  $Y_1\subseteq   X_{11} \cap X_{21}$ and 
$Y_2\subseteq   X_{12} \cap X_{22}$ or    $Y_1\subseteq   X_{11} \cap X_{22}$ and 
$Y_2\subseteq   X_{12} \cap X_{21}$. 
\end{theorem}

This clearly implies the first part of (\ref{idezet}), as every $n$-leaf phylogenetic tree contains an edge whose removal partitions the tree into two components, each containing at least $n/3$ of the leaves. This fact has been rediscovered many times; the earliest reference we are aware of is \cite{EFRS}. In fact, both \cite{spencer} and \cite{sagi} proved,  mutatis mutandis, the following Theorem~\ref{measure}, if we replace cardinalities by measures in the proof:

\begin{theorem} \label{measure}
Let $(M,\mu)$ be measure space such that $\mu(M)<\infty$. Let $P_1,P_2$ be arbitary bipartitions of $M$ such that partition classes are measurable and every partition class has measure at least $\mu(M)/3$. Then there exist 
measurable sets $A,B$, such that
$A$ and $B$ are subsets of different partition classes of $P_1$ and $P_2$, and $\mu(A)\geq \mu(M)/6$ and $\mu(B)\geq \mu(M)/6$.  
\end{theorem}

The second part of the displayed formula  (\ref{idezet}) is substantially different, as it does not hold for {\em every} choice of three $1/3$-$2/3$ bipartitions, one selected from each of the three phylogenetic trees. This is already shown by the three bipartitions of the  unique 3-leaf phylogenetic tree.

\section{A common quartet of many phylogenetic trees}
Given $k$,  Snir and Yuster \cite{sagi} ask for the smallest $n$ such that any $k$ phylogenetic 
trees on the same set of $n$ leaves share a quartet. They denote this $n$ by $h(k)$ and show that
$10\leq h(3)\leq 14$. The upper bound follows from (\ref{idezet}), as $\lceil 2\cdot 14/27\rceil=2$.
Formula \eqref{qsplit _on_6} means that $h(2)\leq 6$, and it is not difficult to see that $h(2)= 6$.
We observe that the function $h(k)$ is defined for every $k\geq 2$, as $h(k)=M_k(4)$. 

\begin{theorem} For $k\geq 3$, we have
\begin{equation}
8^{1/4}3^{k/4} \leq h(k)   \leq 
\tw_3\Bigg(6(k-1) 2^k +O(\log k )  \Bigg),
\end{equation}
where the $O(\cdot)$-term is with respect to $k\rightarrow\infty$. 
\end{theorem}
\begin{proof}
Based on (\ref{mainineq}), we have 
$$ h(k)=M_k(4)\leq M_k(6)\leq R^4_{2^{k-1}}(6)\leq 
\tw_3\Bigg(6(k-1) 2^k +O(\log k )  \Bigg).
$$

For the lower bound, consider $k$ independently and uniformly selected phylogenetic trees on the same set of $n$ leaves.
Fix four leaves $a,b,c,d$. The probability that these leaves induce the same  quartet split in all $k$ trees is ${3 \cdot \left(\frac{1}{3}\right)^k=}3^{1-k}$.
By linearity of expectation, the expected number of quartets that are resolved identically in all $k$ trees is
$\binom{n}{4}3^{1-k}$. 
If $n\leq 8^{1/4}3^{k/4}$, this expectation is less than 1. Consequently, there exists an outcome of the random experiment with no quartet that is resolved identically in all $k$ trees.
\end{proof}

Although hypergraph Ramsey problems typically have lower bounds given by a tower of height one less than that of the corresponding upper bounds, such lower bounds do not automatically carry over to the Maximum Agreement Subtree problem.

We find it curious that, in order to find a common quartet split among trees, we actually need to first guarantee a common induced 6-leaf binary tree.

\end{document}